\documentclass[reqno]{amsart}

\usepackage[latin1]{inputenc}
\usepackage{amssymb}
\usepackage{graphicx}
\usepackage{amscd}
\usepackage[hidelinks]{hyperref}
\usepackage{color}
\usepackage{float}
\usepackage{graphics,amsmath,amssymb}
\usepackage{amsthm}
\usepackage{amsfonts}
\usepackage{latexsym}
\usepackage{epsf}
\usepackage{enumerate}
\usepackage{xifthen}
\usepackage{mathrsfs}
\usepackage{dsfont}
\usepackage{makecell}
\usepackage[FIGTOPCAP]{subfigure}
\usepackage{amsmath}
\allowdisplaybreaks[4]
\usepackage{listings}
\usepackage{etoolbox}
\usepackage{fancyhdr}

 \newtheoremstyle{mytheorem}
 {3pt}
 {3pt}
 {\slshape}
 {}
 {\bfseries}
 {.}
 { }
 {}

\numberwithin{equation}{section}

\theoremstyle{theorem}
\newtheorem{theorem}{Theorem}[section]

\newtheorem{lemma}[theorem]{Lemma}

\theoremstyle{definition}

\newcommand{\Keywords}[1]{\ifthenelse{\isempty{#1}}{}{\smallskip \smallskip \noindent \textbf{Keywords}. #1}}
\newcommand{\MSC}[2][2010]{\ifthenelse{\isempty{#2}}{}{\smallskip \smallskip \noindent \textbf{#1MSC}. #2}}
\newcommand{\abstractnote}[1]{\ifthenelse{\isempty{#1}}{}{\smallskip \smallskip \noindent \textsuperscript{\dag}#1}}

\makeatletter
\def\specialsection{\@startsection{section}{1}%
  \z@{\linespacing\@plus\linespacing}{.5\linespacing}%
  {\normalfont}}
\def\section{\@startsection{section}{1}%
  \z@{.7\linespacing\@plus\linespacing}{.5\linespacing}%
  {\normalfont\scshape}}
\patchcmd{\@settitle}{\uppercasenonmath\@title}{\Large\boldmath}{}{}
\patchcmd{\@settitle}{\begin{center}}{\begin{flushleft}}{}{}
\patchcmd{\@settitle}{\end{center}}{\end{flushleft}}{}{}
\patchcmd{\@setauthors}{\MakeUppercase}{\normalsize}{}{}
\patchcmd{\@setauthors}{\centering}{\raggedright}{}{}
\patchcmd{\section}{\scshape}{\large\bfseries\boldmath}{}{}
\patchcmd{\subsection}{\bfseries}{\bfseries\boldmath}{}{}
\renewcommand{\@secnumfont}{\bfseries}
\patchcmd{\@startsection}{\@afterindenttrue}{\@afterindentfalse}{}{}
\patchcmd{\abstract}{\leftmargin3pc}{\leftmargin1pc}{}{}

\def\maketitle{\par
  \@topnum\z@ 
  \@setcopyright
  \thispagestyle{empty}
  \ifx\@empty\shortauthors \let\shortauthors\shorttitle
  \else \andify\shortauthors
  \fi
  \@maketitle@hook
  \begingroup
  \@maketitle
  \toks@\@xp{\shortauthors}\@temptokena\@xp{\shorttitle}%
  \toks4{\def\\{ \ignorespaces}}
  \edef\@tempa{%
    \@nx\markboth{\the\toks4
      \@nx\MakeUppercase{\the\toks@}}{\the\@temptokena}}%
  \@tempa
  \endgroup
  \c@footnote\z@
  \@cleartopmattertags
}
\makeatother

\newcommand{\op}{\overline{p}}

\title[Congruences for overpartition]{Some congruences modulo $5$ and $25$ for overpartition}

\author[S. Chern]{Shane Chern}
\address[S. Chern]{Department of Mathematics, Pennsylvania State University, University Park, PA 16802, USA}
\email{shanechern@psu.edu; chenxiaohang92@gmail.com}

\author[M. G. Dastidar]{Manosij Ghosh Dastidar}
\address[M. G. Dastidar]{Department of Mathematical Sciences, Pondicherry University, R. V. Nagar, Kalapet, Puducherry, PIN-605014, India}
\email{gdmanosij@gmail.com}

\date{}

\begin{document}

\maketitle

\begin{abstract}

We present two new Ramanujan-type congruences modulo $5$ for overpartition. We also give an affirmative answer to a conjecture of Dou and Lin, which includes four congruences modulo $25$ for 
overpartition.


\MSC{Primary 11P83; Secondary 05A17.}
\end{abstract}

\section{Introduction}\label{sect:1}

An \textit{overpartition} of a positive integer $n$ is a partition of $n$ where the first occurrence of each distinct part may be overlined. For example, $3$ has the following eight overpartitions:
$$3,\ \overline{3},\ 2+1,\ \overline{2}+1,\ 2+\overline{1},\ \overline{2}+\overline{1},\ 1+1+1,\ \overline{1}+1+1.$$

As usual, we use $\op(n)$ to denote the number of overpartitions of $n$. We also agree that $\op(0)=1$. It is well known that the generating function of $\op(n)$ is
$$\sum_{n\ge 0}\op(n) q^n= \frac{(-q;q)_\infty}{(q;q)_\infty}=\frac{(q^2;q^2)_\infty}{(q;q)_\infty^2},$$
where we adopt the standard notation
$$(a;q)_\infty=\prod_{n\ge 0}(1-aq^n).$$
For convenience, when $k$ is a positive integer, we write
$$E_k=E_k(q):=(q^k;q^k)_\infty.$$

Similar to the standard partition function $p(n)$, many authors also found various Ramanujan-type identities and congruences for $\op(n)$. The interested readers may refer to \cite{CSWZ2015, HS2005, Kim2008, Mah2004, YX2013}. Recently, Dou and Lin \cite{DL2016} presented the following four congruences for $\op(n)$:
$$\op(80n+t)\equiv 0 \pmod{5},$$
where $t=8$, $52$, $68$, and $72$. In a subsequent paper, Hirschhorn \cite{Hir2016} obtained simple proofs of Dou and Lin's congruences.

The main purpose of this paper is to give an elementary proof to a conjecture of Dou and Lin \cite{DL2016}.

\begin{theorem}\label{th:1}
For $n\ge 0$,
\begin{align}
\op(80n+8)&\equiv 0 \pmod{25},\label{eq:th1.1}\\
\op(80n+52)&\equiv 0 \pmod{25},\label{eq:th1.2}\\
\op(80n+68)&\equiv 0 \pmod{25},\label{eq:th1.3}\\
\op(80n+72)&\equiv 0 \pmod{25}.\label{eq:th1.4}
\end{align}
\end{theorem}

Furthermore, using a powerful method involving modular forms due to Radu and Sellers \cite{RS2011}, we also obtain two new congruences modulo $5$ for $\op(n)$.

\begin{theorem}\label{th:2}
For $n\ge 0$,
\begin{equation}\label{eq:th2}
\op(135n+t)\equiv 0 \pmod{5},
\end{equation}
where $t=63$ and $117$.
\end{theorem}

\section{Proof of Theorem \ref{th:1}}

\subsection{Preliminaries}

We first introduce some notations of Ramanujan's theta functions. Let
$$\varphi(q):=\sum_{n=-\infty}^\infty q^{n^2} \text{ and } \psi(q):=\sum_{n\ge 0}q^{n(n+1)/2}.$$
We can write $\varphi(q)$, $\varphi(-q)$, $\psi(q)$, and $\psi(-q)$ in terms of some $E_k$'s defined in the previous section:
$$\varphi(q)=\frac{E_2^5}{E_1^2 E_4^2},\quad \varphi(-q)=\frac{E_1^2}{E_2},\quad \psi(q)=\frac{E_2^2}{E_1},\quad \psi(-q)=\frac{E_1 E_4}{E_2}.$$
We also write
$$\chi(q)=\frac{E_2^2}{E_1 E_4},\quad \chi(-q)=\frac{E_1}{E_2},\quad f(-q)=E_1.$$

At first, we notice that from the binomial theorem,
$$E_k^{25}\equiv E_{5k}^5 \pmod{25}$$
holds for any positive integer $k$. The following results also play an important role in our proof of Theorem \ref{th:1}.

\begin{lemma}\label{le:2.1}
\begin{equation}\label{eq:le2.1}
\frac{E_5^2}{E_{10}^4}\equiv \frac{E_1^{26}}{E_2^{28}}+q \frac{E_1^2}{E_2^4} \pmod{25}.
\end{equation}
\end{lemma}

\begin{proof}
Ramanujan \cite[p. 365, Eq. (18.1)]{Ber1998} asserted that
$$\psi(q)^2-q\psi(q^5)^2=\frac{f(-q^5)\varphi(-q^5)}{\chi(-q)},$$
that is,
$$\frac{E_2^4}{E_1^2}-q\frac{E_{10}^4}{E_5^2}=\frac{E_2 E_5^3}{E_1 E_{10}}.$$
Hence we have
\begin{align*}
\frac{E_5^2}{E_{10}^4}&=\frac{E_1^2}{E_2^4}\left(\frac{E_2 E_5^5}{E_1 E_{10}^5}+q\right)\equiv \frac{E_1^2}{E_2^4}\left(\frac{E_2 E_1^{25}}{E_1 E_{2}^{25}}+q\right)=\frac{E_1^{26}}{E_2^{28}}+q \frac{E_1^2}{E_2^4} \pmod{25}.
\end{align*}
\end{proof}

\begin{lemma}\label{le:2.2}
\begin{equation}\label{eq:le2.2}
\frac{1}{E_5 E_{20}}\equiv \frac{E_2^{40}}{E_1^{21} E_4^{21}}+4q \frac{E_1^3 E_4^3}{E_2^8} \pmod{25}.
\end{equation}
\end{lemma}

\begin{proof}
Ramanujan \cite[p. 258, Entry 9(iii)]{Ber1991} also asserted that
$$\varphi(q)^2-\varphi(q^5)^2=4q\chi(q)f(-q^5)f(-q^{20}),$$
that is,
$$\frac{E_2^{10}}{E_1^4 E_4^4}-\frac{E_{10}^{10}}{E_5^4 E_{20}^4}=4q\frac{E_2^2 E_5 E_{20}}{E_1 E_{4}}.$$
Hence we have
\begin{align*}
\frac{1}{E_5 E_{20}}&=\frac{E_1^4 E_4^4}{E_2^{10}}\left(\frac{E_{10}^{10}}{E_5^5 E_{20}^5}+4q\frac{E_2^2}{E_1 E_{4}}\right)\\
&\equiv \frac{E_1^4 E_4^4}{E_2^{10}}\left(\frac{E_{2}^{50}}{E_1^{25} E_{4}^{25}}+4q\frac{E_2^2}{E_1 E_{4}}\right)\\
&=\frac{E_2^{40}}{E_1^{21} E_4^{21}}+4q \frac{E_1^3 E_4^3}{E_2^8} \pmod{25}.
\end{align*}
\end{proof}

Let
$$p=p(q):=\frac{\varphi(q)^2-\varphi(q^3)^2}{2\varphi(q^3)^2},$$
and
$$k=k(q):=\frac{\varphi(q^3)^3}{\varphi(q)}.$$
The following relations are due to Alaca and Williams \cite{AW2010}.

\begin{lemma}\label{le:AW}
\begin{equation}
E_1=2^{-\frac{1}{6}}q^{-\frac{1}{24}}p^{\frac{1}{24}}(1-p)^{\frac{1}{2}}(1+p)^{\frac{1}{6}}(1+2p)^{\frac{1}{8}}(2+p)^{\frac{1}{8}}k^{\frac{1}{2}},
\end{equation}
\begin{equation}
E_2=2^{-\frac{1}{3}}q^{-\frac{1}{12}}p^{\frac{1}{12}}(1-p)^{\frac{1}{4}}(1+p)^{\frac{1}{12}}(1+2p)^{\frac{1}{4}}(2+p)^{\frac{1}{4}}k^{\frac{1}{2}},
\end{equation}
and
\begin{equation}
E_4=2^{-\frac{2}{3}}q^{-\frac{1}{6}}p^{\frac{1}{6}}(1-p)^{\frac{1}{8}}(1+p)^{\frac{1}{24}}(1+2p)^{\frac{1}{8}}(2+p)^{\frac{1}{2}}k^{\frac{1}{2}}.
\end{equation}
\end{lemma}

\subsection{Proof of Eqs. (\ref{eq:th1.1}) and (\ref{eq:th1.4})}

We need the following interesting identity.

\begin{lemma}\label{le:2.4}
\begin{equation}\label{eq:le2.4}
\frac{E_2^{30}}{E_{1} E_4^{31}}+24q\frac{E_1^7 E_2^6}{E_4^{15}}+21q^2 \frac{E_1^{15} E_4}{E_2^{18}}\equiv \frac{E_5^3}{E_{10}^2 E_{20}^3} \pmod{25}.
\end{equation}
\end{lemma}

\begin{proof}
Note that
$$\frac{E_5^3}{E_{10}^2 E_{20}^3}=\frac{E_5^2}{E_{10}^4}\cdot E_5 E_{20}\cdot \frac{E_{10}^2}{E_{20}^4}.$$
Hence it suffices to show
$$\left(\frac{E_2^{30}}{E_{1} E_4^{31}}+24q\frac{E_1^7 E_2^6}{E_4^{15}}+21q^2 \frac{E_1^{15} E_4}{E_2^{18}}\right)\frac{1}{E_5 E_{20}}\equiv \frac{E_5^2}{E_{10}^4}\frac{E_{10}^2}{E_{20}^4} \pmod{25}.$$
Thanks to Lemmas \ref{le:2.1} and \ref{le:2.2}, we can rewrite it as
\begin{align*}
0&\equiv \frac{E_2^{70}}{E_1^{22}E_4^{52}}-\frac{E_1^{26}}{E_2^2 E_4^{28}}+24q\frac{E_2^{46}}{E_1^{14} E_4^{36}}+3q \frac{E_1^2 E_2^{22}}{E_4^{28}}+21q^2 \frac{E_2^{22}}{E_1^6 E_4^{20}}\\
&\quad +96q^2\frac{E_1^{10}}{E_2^2 E_4^{12}} -q^2\frac{E_1^{26}}{E_2^{26} E_4^4}-q^3 \frac{E_1^2}{E_2^2 E_4^4}+84q^3 \frac{E_1^{18} E_4^4}{E_2^{26}} \pmod{25},
\end{align*}
or equivalently,
\begin{align*}
0&\equiv \Bigg(\frac{E_2^{70}}{E_1^{22}E_4^{52}}-\frac{E_1^{26}}{E_2^2 E_4^{28}}+24q\frac{E_2^{46}}{E_1^{14} E_4^{36}}+3q \frac{E_1^2 E_2^{22}}{E_4^{28}}+21q^2 \frac{E_2^{22}}{E_1^6 E_4^{20}}\\
&\quad +96q^2\frac{E_1^{10}}{E_2^2 E_4^{12}} -q^2\frac{E_1^{26}}{E_2^{26} E_4^4}-q^3 \frac{E_1^2}{E_2^2 E_4^4}+84q^3 \frac{E_1^{18} E_4^4}{E_2^{26}}\Bigg)E_1^6E_2^2E_4^{20} \pmod{25},
\end{align*}
since the constant term in the series of $E_1^6E_2^2E_4^{20}$ is $1\in(\mathbb{Z}/25\mathbb{Z})^\times$, and hence $E_1^6E_2^2E_4^{20}$ is invertible in the ring $\mathbb{Z}/25\mathbb{Z}[[q]]$. According to Lemma \ref{le:AW}, it becoms
\begin{align*}
0&\equiv \frac{25}{4096}k^{12} p (1-p)^6 (1+p)^2 (2+p)^3R_1(p) \pmod{25},
\end{align*}
where
\begin{align*}
R_1(p)&=768 + 10240 p + 69056 p^2 + 293440 p^3 + 841488 p^4 + 1663680 p^5\\
&\quad + 2253588 p^6+ 1995132 p^7 + 1015539 p^8 + 199890 p^9 - 17692 p^{10}\\
&\quad  + 784 p^{11} - 4 p^{12}.
\end{align*}
Hence the identity follows obviously.
\end{proof}

Now from \cite[Eq. (3.4)]{Hir2016}, we deduce that
\begin{align*}
\sum_{n\ge 0}\op(16n+8)q^n&\equiv 3qE_1\frac{E_2^{25} E_4^{25}}{E_1^{50}}\left(\frac{E_2^{30}}{E_{1} E_4^{31}}+24q\frac{E_1^7 E_2^6}{E_4^{15}}+21q^2 \frac{E_1^{15} E_4}{E_2^{18}}\right)\\
&\equiv 3qE_1\frac{E_{10}^{5} E_{20}^{5}}{E_5^{10}} \frac{E_5^3}{E_{10}^2 E_{20}^3}\\
&=3qE_1\frac{E_{10}^{3} E_{20}^{2}}{E_5^{7}} \pmod{25}.
\end{align*}
Note that
$$E_1=\sum_{n=-\infty}^\infty (-1)^n q^{(3n^2+n)/2}$$
has no terms of the form $q^{5n+3}$ and $q^{5n+4}$, while $\frac{E_{10}^{3} E_{20}^{2}}{E_5^{7}}$ is a series of $q^5$. Hence
$$\op(16(5n+4)+8)=\op(80n+72)\equiv 0 \pmod{25},$$
and
$$\op(16(5n)+8)=\op(80n+8)\equiv 0 \pmod{25}.$$

\subsection{Proof of Eqs. (\ref{eq:th1.2}) and (\ref{eq:th1.3})}

This time we need
\begin{lemma}\label{le:2.5}
\begin{equation}\label{eq:le2.5}
\frac{E_2^{10}}{E_{1}^{11} E_4}+14q\frac{E_4^{15}}{E_1^3 E_2^{14}}+16q^2 \frac{E_1^{5} E_4^{31}}{E_2^{38}}\equiv \frac{E_5 E_{20}^3}{ E_{10}^6} \pmod{25}.
\end{equation}
\end{lemma}

\begin{proof}
Note that
$$\frac{E_5 E_{20}^3}{ E_{10}^6}=\frac{E_5^2}{E_{10}^4} \cdot \frac{1}{E_5 E_{20}} \cdot\frac{E_{20}^4}{E_{10}^2}.$$
It therefore suffices to show
$$\left(\frac{E_2^{10}}{E_{1}^{11} E_4}+14q\frac{E_4^{15}}{E_1^3 E_2^{14}}+16q^2 \frac{E_1^{5} E_4^{31}}{E_2^{38}}\right)\frac{E_{10}^2}{E_{20}^4}\equiv \frac{E_5^2}{E_{10}^4}  \frac{1}{E_5 E_{20}} \pmod{25}.$$
Again by Lemmas \ref{le:2.1} and \ref{le:2.2}, we can rewrite it as
\begin{align*}
0&\equiv \frac{E_2^{36}}{E_1^{11}E_4^{29}}-\frac{E_1^{5}E_2^{12}}{E_4^{21}}-q\frac{E_2^{36}}{E_1^{19} E_4^{21}}+14q \frac{E_2^{12}}{E_1^3 E_4^{13}}-4q \frac{E_1^{29} E_4^{3}}{E_2^{36}}\\
&\quad +q^2\frac{E_2^{12}}{E_1^{11} E_4^{5}} +12q^2\frac{E_1^{5}E_4^3}{E_2^{12}}+14q^3 \frac{E_4^{11}}{E_1^3 E_2^{12}}+16q^4 \frac{E_1^{5} E_4^{27}}{E_2^{36}} \pmod{25},
\end{align*}
or equivalently,
\begin{align*}
0&\equiv \Bigg(\frac{E_2^{36}}{E_1^{11}E_4^{29}}-\frac{E_1^{5}E_2^{12}}{E_4^{21}}-q\frac{E_2^{36}}{E_1^{19} E_4^{21}}+14q \frac{E_2^{12}}{E_1^3 E_4^{13}}-4q \frac{E_1^{29} E_4^{3}}{E_2^{36}}\\
&\quad +q^2\frac{E_2^{12}}{E_1^{11} E_4^{5}} +12q^2\frac{E_1^{5}E_4^3}{E_2^{12}}+14q^3 \frac{E_4^{11}}{E_1^3 E_2^{12}}+16q^4 \frac{E_1^{5} E_4^{27}}{E_2^{36}}\Bigg)E_1^3E_4^{13} \pmod{25}.
\end{align*}
According to Lemma \ref{le:AW}, it becoms
\begin{align*}
0&\equiv \frac{25 k^6 p (2+p)^3}{4096 (1+2 p)^3}R_2(p) \pmod{25},
\end{align*}
where
\begin{align*}
R_2(p)&=256 + 3200 p + 9536 p^2 + 7264 p^3 - 35088 p^4 - 94272 p^5 - 
 66744 p^6\\
&\quad + 20184 p^7 + 22830 p^8 - 24652 p^9 - 19400 p^{10} - 
 302 p^{11} + 41 p^{12}.
\end{align*}
The lemma follows immediately.
\end{proof}

Similarly we see from \cite[Eq. (3.3)]{Hir2016} that
\begin{align*}
\sum_{n\ge 0}\op(16n+4)q^n&\equiv 14E_1\frac{E_2^{75}}{E_1^{50} E_4^{25}}\left(\frac{E_2^{10}}{E_{1}^{11} E_4}+14q\frac{E_4^{15}}{E_1^3 E_2^{14}}+16q^2 \frac{E_1^{5} E_4^{31}}{E_2^{38}}\right)\\
&\equiv 14E_1\frac{E_{10}^{5} E_{20}^{5}}{E_5^{10}} \frac{E_5 E_{20}^3}{ E_{10}^6}\\
&=14E_1\frac{E_{20}^{8}}{E_{5}^{9}E_{10}} \pmod{25}.
\end{align*}
Through a similar argument, we conclude that
$$\op(16(5n+3)+4)=\op(80n+52)\equiv 0 \pmod{25},$$
and
$$\op(16(5n+4)+4)=\op(80n+68)\equiv 0 \pmod{25}.$$

\section{Proof of Theorem \ref{th:2}}

\subsection{The method of Radu and Sellers}

Let $\Gamma:=\mathrm{SL}_2(\mathbb{Z})$. For a positive integer $N$, we define
$$\Gamma_0(N)=\left\{\left.\begin{pmatrix}a & b\\c & d\end{pmatrix}\ \right|\ c\equiv 0\pmod{N}\right\},$$
and
$$\Gamma_\infty=\left\{\left.\begin{pmatrix}1 & h \\0 & 1 \end{pmatrix}\ \right|\ h\in\mathbb{Z}\right\}.$$

Let $M$ be a positive integer. We denote by $R(M)=\{r:r=(r_{\delta_1},\ldots,r_{\delta_k})\}$ the set of integer sequences indexed by the positive divisors $1=\delta_1<\cdots<\delta_k=M$ of $M$. For a positive integer $m$, we set $[s]_m=s+m\mathbb{Z}$. We also denote by $\mathbb{Z}_m^*$ the set of all invertible elements in $\mathbb{Z}_m$, and by $\mathbb{S}_m$ the set of all squares in $\mathbb{Z}_m^*$. For $t\in\{0,\ldots,m-1\}$, we set
$$P_{m,r}(t)=\left\{t'\ |\ t'\equiv ts+\frac{s-1}{24}\sum_{\delta\mid M}\delta r_{\delta}\pmod{m},0\le t'\le m-1,[s]_{24m}\in\mathbb{S}_{24m}\right\},$$
$$p_{m,r}(\gamma)=\min_{\lambda\in\{0,\ldots,m-1\}}\frac{1}{24}\sum_{\delta\mid M}r_\delta\frac{\gcd^2(\delta(a+\kappa\lambda c),mc)}{\delta m},$$
and
$$p_{r'}^*(\gamma)=\frac{1}{24}\sum_{\delta\mid N} \frac{r'_\delta\gcd^2(\delta,c)}{\delta},$$
where $\gamma=\begin{pmatrix}a & b \\c & d \end{pmatrix}$, $r\in R(M)$, $r'\in R(N)$, and $\kappa=\kappa(m)=\gcd(m^2-1,24)$.

Finally let
$$f_r(q):=\prod_{\delta\mid M}(q^{\delta};q^{\delta})_{\infty}^{r_\delta}=\sum_{n\ge 0}c_r(n)q^n$$
for some $r\in R(M)$. Radu and Sellers' result states as

\begin{lemma}\label{le:RS}
Let $u$ be a positive integer, $(m,M,N,t,r=(r_\delta))\in\Delta^*$, the set of tuples satisfying conditions given in \cite[p. 2255]{RS2011}, $r'=(r'_\delta)\in R(N)$, $n$ be the number of double cosets in $\Gamma_0(N)\backslash\Gamma/\Gamma_\infty$, and $\{\gamma_1,\ldots,\gamma_n\}$ $\subset\Gamma$ be a complete set of representatives of the double coset $\Gamma_0(N)\backslash\Gamma/\Gamma_\infty$. Assume that $p_{m,r}(\gamma_i)+p_{r'}^*(\gamma_i)\ge 0$ for all $i=1,\ldots,n$. Let $t_{\min} := \min_{t'\in P_{m,r}(t)}t'$ and
$$v:=\frac{1}{24}\left(\left(\sum_{\delta\mid M}r_\delta+\sum_{\delta\mid N}r'_\delta\right)\left(N\prod_{p\mid N}(1+p^{-1})\right)-\sum_{\delta\mid N}\delta r'_\delta\right)-\frac{1}{24m}\sum_{\delta\mid M}\delta r_\delta-\frac{t_{\min}}{m}.$$
Then if
$$\sum_{n=0}^{\lfloor v \rfloor}c_r(mn+t')q^n\equiv 0 \pmod{u},$$
for all $t'\in P_{m,r}(t)$, then
$$\sum_{n\ge 0}c_r(mn+t')q^n\equiv 0 \pmod{u},$$
for all $t'\in P_{m,r}(t)$.
\end{lemma}

\subsection{Proof of Theorem \ref{th:2}}

To apply the method of Radu and Sellers, we notice that
$$\frac{E_2}{E_1^2}\equiv \frac{E_1^{3} E_2}{E_5} \pmod{5}.$$

Then we may take
$$(m,M,N,t,r=(r_1,r_2,r_{5},r_{10}))=(135,10,30,63,(3,1,-1,0))\in\Delta^*.$$
By the definition of $P_{m,r}(t)$, we have
$$P_{m,r}(t)=\left\{t'\ |\ t'\equiv ts\ (\bmod\ m),0\le t'\le m-1,[s]_{24m}\in\mathbb{S}_{24m}\right\}=\{63,117\}.$$
Now we can choose
$$r'=(r'_1,r'_2,r'_3,r'_5,r'_{6},r'_{10},r'_{15},r'_{30})=(7, 2, 0, 0, 0, 1, 0, 0).$$
Let
$$\gamma_\delta=\begin{pmatrix}
1 & 0 \\
\delta & 1 
\end{pmatrix}.$$
It follows by \cite[Lemma 2.6]{RS2011} that $\{\gamma_\delta:\delta\mid N\}$ contains a complete set of representatives of the double coset $\Gamma_0(N)\backslash\Gamma/\Gamma_\infty$. One may verify that all assumptions of Lemma \ref{le:RS} are satisfied. Hence we obtain the upper bound $\lfloor v \rfloor=37$. Theorem \ref{th:2} follows immediately by checking the two congruences for $n$ from $0$ to $37$.

\subsection{Remarks}

We remark that this method can also be applied to prove Theorem \ref{th:1}. At first, note that
$$\frac{E_2}{E_1^2}\equiv \frac{E_1^{23} E_2}{E_5^5} \pmod{25}.$$

Now for $\op(80n+8)$ and $\op(80n+72)$, we choose
$$(m,M,N,t,r=(r_1,r_2,r_{5},r_{10}))=(80,10,40,8,(23,1,-5,0))\in\Delta^*,$$
and compute $P_{m,r}(t)=\{8,72\}$. Then we set
$$r'=(r'_1,r'_2,r'_4,r'_5,r'_{8},r'_{10},r'_{20},r'_{40})=(5, 0, 0, 0, 0, 0, 0, 0),$$
and therefore obtain the upper bound $\lfloor v \rfloor=71$.

Similarly, for $\op(80n+52)$ and $\op(80n+68)$, we have
$$(m,M,N,t,r=(r_1,r_2,r_{5},r_{10}))=(80,10,40,52,(23,1,-5,0))\in\Delta^*,$$
$$P_{m,r}(t)=\{52,68\},\quad r'=(r'_1,r'_2,r'_4,r'_5,r'_{8},r'_{10},r'_{20},r'_{40})=(5, 0, 0, 0, 0, 0, 0, 0),$$
and thus the upper bound $\lfloor v \rfloor=71$.

Finally, we see that Theorem \ref{th:1} follows directly from a simple verification.

\subsection*{Acknowledgements}

We thank George E. Andrews and Yucheng Liu for some helpful discussions.

\bibliographystyle{amsplain}

\end{document}